\documentclass[10pt]{amsart}
\usepackage{graphicx,amscd,color,amsmath,amsfonts,amssymb,geometry,amssymb,amsthm}
\usepackage[initials]{amsrefs}
\usepackage{graphicx}
\newtheorem{theorem}{Theorem}[section]
\newtheorem{proposition}[theorem]{Proposition}
\newtheorem{corollary}[theorem]{Corollary}

\newtheorem{remark}[theorem]{Remark}
\newtheorem{lemma}[theorem]{Lemma}

\numberwithin{equation}{section}
\input epsf

\newcommand{\R}{\mathbb{R}}

\newcommand{\Q}{\mathbb{Q}}
\newcommand{\Z}{\mathbb{Z}}

\renewcommand{\epsilon}{\varepsilon}
\renewcommand{\phi}{\varphi}

\begin{document}

\title{On inequalities for convex functions.}


\author[L. Bernal]{Luis Bernal-Gonz\'alez}
\author[P. Jim\'enez]{Pablo Jim\'enez-Rodr\'iguez}
\author[G. Mu\~noz]{Gustavo A. Mu\~noz-Fern\'andez}
\author[J. Seoane]{Juan B. Seoane-Sep\'ulveda.}

\begin{abstract}
We study some properties convex functions fulfill. Among the conclusions we obtain from such result, we are able to prove some nontrivial inequalities among real numbers, and we give an improvement of the reverse triangle inequality in the particular case where three aligned points are considered in a normed space. 
\end{abstract}

\maketitle

\section{Introduction}

Convexity is one of the first concepts that are introduced in an undergraduate course of mathematics, and its definition is quite natural to state. Nevertheless, convex functions keep playing a central role in operator theory, in real analysis and in some realms of applied mathematics, like management science or optimization theory (we refer the interested reader to \cite{Ng,Ni} for applications of the property of convexity). \\
We remind that a function $f:V \rightarrow \R$ (where $V$ is a vector space over $\R$) is {\bf convex} if, whenever $x, \, y \in V$ and $0 \leq \lambda \leq 1$, then
$$
f(\lambda x+(1-\lambda)y) \leq \lambda f(x)+(1-\lambda)f(y).
$$
Similarly, we say that $f$ is {\bf midconvex} if, whenever $x, \, y \in V$,
$$
f\left({x+y \over 2}\right) \leq {1 \over 2}f(x)+{1 \over 2}f(y).
$$
Jensen started the systematic study of convex functions, where his interest was focused on the property of midconvexity (\cite{Jen1, Jen2}). It is trivial that a convex function is midconvex as well, and one natural question is whether midconvexity is enough to guarantee convexity. In this direction, the answer is well-known:
\begin{theorem}[\cite{RoVar}, p. 217] \label{counterexample} There exist additive functions (that is, functions that satisfy $f(x+y)=f(x)+f(y)$ for every $x,y \in V$) that are discontinuous on $\R$. Remark that an additive function satisfies $f(px+qy)=pf(x)+qf(y)$ for every $x,y \in V$ and $p,q \in \Q$.
\end{theorem} 
The existence of such function is easy to find: consider a Hamel basis of $\R$ over $\Q$, $\{r_i\}_{i \in I}$, and define an arbitrary function $f:\{r_i\}_{i \in I} \rightarrow \R$ with the only condition that it is not linear. We can then obtain the desirable function by extending $f$ by linearity. \\
In fact, one can find a vector space of dimension $2^{\mathfrak{c}}$ every nonzero element of which is an everywhere discontinuous additive function (see, for example, \cite{ArBerPeSe}). \\
The search of large algebraic structures whose nonzero elements fulfill some kind of surprising property has been a very fruitful field in the past years, even though it is not the main interest of the present manuscript. We refer the interested reader to the monographs \cite{BerPeSe,ArBerPeSe} and the references therein. \\
One of the interests in the theory of convex functions is to give non-trivial characterizations of convexity. For example, even though we have already seen that a midconvex function is not convex, a midconvex function which is bounded above in a neighborhood of a single point $x_0 \in V$ is continuous, and hence convex (\cite{RoVar}, p. 215). Having that a convex function is continuous, we could have the following well-known proposition:
\begin{proposition} \label{conv_midconv}
A function $f:V \rightarrow \R$ is convex if and only if it is midconvex and it is bounded above in a neighborhood of a single point $x_0 \in V$.
\end{proposition}

The study of characterizations of the property of convexity has been fruitful in the past years, and numerous papers have been published. For example, Nikodem proved that the boundedness condition in Proposition \ref{conv_midconv} could be substituted by quasiconvexity (that is, $f(\lambda x+(1-\lambda y)) \leq \max\{f(x),f(y)\}$ for every $x, \, y \in V$ and $0 \leq \lambda \leq 1$). Teo and Yang could extend the previous characterization to {\it any} value $\lambda \in (0,1)$:
\begin{theorem}[\cite{TeoYang}]
$f:V \rightarrow \R$ is convex if and only if $f$ is quasiconvex and there exists $\lambda_0 \in (0,1)$ so that
$$
f(\lambda_0x+(1-\lambda_0)y) \leq \lambda_0f(x)+(1-\lambda_0)f(y).
$$
\end{theorem}
Recently, Fujii, Kim and Nakamoto characterized the property of convexity in terms of the average rate of change:
\begin{theorem}[\cite{FuKimNa}]
A real valued function $f$ on an interval $J=[a,b)$ with $b \in (-\infty,+\infty]$ is convex (resp. concave) if and only if, for each $0<\epsilon<b-a$, $f(t+\epsilon)-f(t)$ is non-decreasing (resp. non-increasing) on $[a,b-\epsilon)$.
\end{theorem}
Using that characterization, the authors are able to prove some non-trivial theorems concerning operator monotone functions. For some more recent characterizations, the interested reader can also consult the results in \cite{E,HanMart,LiYeh,Mi,VoHiJe}. \\
In this article we will give a property for convex functions that, similarly to midconvexity, gives almost a characterization (resulting an equivalent property in the situation of boundedness over one single interval). Using that property, in section \ref{ineq} we will prove some other inequalities that convex functions fulfill and we will also give several applications in the form of nontrivial inequalities. Finally, in section \ref{multivariable} we will study what can be said for the multivariable case.  \\
For the results that we will present in the following section, we will make use of the following lemma, which is a well-known result:

\begin{lemma} \label{lem1}
If $\phi$ and $\psi$ are two convex functions and $\phi$ is increasing, then $\phi \circ \psi$ is convex as well. \\
Similarly, if $\phi$ and $\psi$ are two concave functions and $\phi$ is increasing, then $\phi \circ \psi$ is concave as well.
\end{lemma}

Remark that, in general, if $\phi$ is not increasing and $\phi$ is convex or concave, we do not obtain that $\phi \circ \psi$ is convex or concave. \\
To show this statement, it suffices to take $\phi(x)=x^2$ (if we want a non increasing convex function) or $\phi(x)=x-x^2$ (if we want a non increasing concave function), and $\psi(x)=x^2+x$ (for a convex function) or $\psi(x)=1-x^2$ (for a concave function). \\
The following proposition has a straightforward proof, which we include for the sake of completeness:

\begin{proposition} \label{eq_conv}
Let $\phi:I \rightarrow \R$ be a convex function and assume that $a<b$ are points in $I$ and that $0<\lambda_0<1$ is such that $\phi(\lambda_0 a+(1-\lambda_0)b)=\lambda_0\phi(a)+(1-\lambda_0)\phi(b)$. Then, $\phi(\lambda a+(1-\lambda)b)=\lambda\phi(a)+(1-\lambda)\phi(b)$ for every $0 \leq \lambda \leq 1$.
\end{proposition}

\begin{proof}
Assume we can find $0<\lambda<1$ so that $\phi(\lambda a+(1-\lambda)b)<\lambda\phi(a)+(1-\lambda)\phi(b)$. If $\lambda>\lambda_0$, write $\lambda_0a+(1-\lambda_0)b={\lambda_0 \over \lambda}(\lambda a+(1-\lambda)b)+\left(1-{\lambda_0 \over \lambda}\right)b$, so that
$$
\begin{aligned}
\lambda_0\phi(a)+(1-\lambda_0)\phi(b)&=\phi(\lambda_0 a+(1-\lambda_0)b)=\phi\left({\lambda_0 \over \lambda}(\lambda a+(1-\lambda)b)+\left(1-{\lambda_0 \over \lambda}\right)b\right) \\
&\leq {\lambda_0 \over \lambda} \phi(\lambda a+(1-\lambda)b)+\left(1-{\lambda_0 \over \lambda}\right)\phi(b) \\
&<{\lambda_0 \over \lambda}(\lambda \phi(a)+(1-\lambda)\phi(b))+\left(1-{\lambda_0 \over \lambda}\right)\phi(b)=\lambda_0\phi(a)+(1-\lambda_0)\phi(b),
\end{aligned}
$$
reaching a contradiction. The case $\lambda<\lambda_0$ is treated analogously, working on the identity
$$
\lambda_0a+(1-\lambda_0)b=\left(1-{1-\lambda_0 \over 1-\lambda}\right)a+{1-\lambda_0 \over 1-\lambda}(\lambda a+(1-\lambda)b).
$$ 
\end{proof}

\section{The main theorem} \label{main_sec}

In this section we will prove the main inequality of this manuscript.
 
\begin{theorem} \label{main}
Let $I\subseteq \R$ be an interval (possibly infinite) and $\phi:I \rightarrow \mathbb{R}$ be a convex function. \\
Then, for every $c,\,a,\,b \in I$ so that $c \leq a,b$ (resp. $c \geq a,b$) and $a+b-c \in I$, we have that 
\begin{equation} \label{ineq_main}
\phi(a)+\phi(b)\leq \phi(c)+\phi(a+b-c).
\end{equation}
If, furthermore, $\phi$ is bounded above in one single interval, then inequality \eqref{ineq_main} whenever $c \leq a,b$ or $c \geq a,b$ is equivalent to convexity.
\end{theorem}

\begin{proof}
Notice that, without loss of generality, we may take $b \geq a$. Assume first $c=0 \leq a,b$. Then, notice that
$$
\phi(a)=\phi\left({a \over a+b}(a+b)+{b \over a+b}\,0\right) \leq {a \over a+b}\phi(a+b)+{b \over a+b}\phi(0) \quad (*).
$$
Therefore, we can write
$$
\begin{aligned}
\phi(b) &\leq \phi\left({a \over b}a+{b-a \over b}(a+b) \right) \leq {a \over b}[\phi(a)-\phi(a+b)]+\phi(a+b) \quad (**) \\ &\stackrel{(*)}{\leq} {a \over b}\left[{a \over a+b}\phi(a+b)+{b \over a+b}\phi(0)-\phi(a+b)\right]+\phi(a+b) \\
&={a \over b}\left[{b \over a+b}\phi(0)-{b \over a+b}\phi(a+b) \right]+\phi(a+b) \\
&={a \over a+b}\phi(0)-{a \over a+b}\phi(a+b)+\phi(a+b)={a \over a+b}\phi(0)+{b \over a+b}\phi(a+b).
\end{aligned}
$$
Therefore,
$$
\phi(a)+\phi(b) \leq {a \over a+b}\phi(a+b)+{b \over a+b}\phi(0)+{a \over a+b}\phi(0)+{b \over a+b}\phi(a+b)=\phi(0)+\phi(a+b).
$$
Assuming now that $c$ is any number less than $a$ and $b$, just consider the convex function $\phi_c=\phi(\cdot+c)$ and notice that $a-c, b-c \geq 0$. Then,
$$
\phi_c(a-c)+\phi_c(b-c) \leq \phi_c(0)+\phi_c(a+b-2c),
$$
and the result follows. \\
Assume finally that $c \geq a,b$ and consider the function $f(x)=\phi(-x)$, which is convex over $-I=\{-t \, : \, t \in I\}$. Then, we know that
$$
f(-a)+f(-b) \leq f(-c)+f(-a-b+c),
$$
and the inequality follows. \\
Let us assume now that $\phi$ satisfies $\phi(a)+\phi(b) \leq \phi(c)+\phi(a+b-c)$ whenever $c \leq a,b$ or $c \geq a,b$. Then, $\phi\left({x+y \over 2}\right)+\phi\left({x+y \over 2}\right) \leq \phi(x)+\phi\left({x+y \over 2}+{x+y \over 2}-x\right)=\phi(x)+\phi(y)$, so that $\phi$ is a midconvex function and therefore $\phi$ is convex (by Proposition (1.2)).
\end{proof}

\begin{remark} \label{remark} \textcolor{white}{hola} \\
\begin{enumerate}
\item Notice that if $I \subseteq \R$ is an interval, $a,b,c \in I$ satisfy $a \leq c \leq b$, we take $\lambda \in [0,1]$ so that $c=\lambda a+(1-\lambda)b$ and $\phi:I \rightarrow \R$ is convex, then
$$
\begin{aligned}
\phi(c)+\phi(a+b-c)&=\phi(\lambda a+(1-\lambda)b)+\phi(a+b-(\lambda a+(1-\lambda)b)) \\
&\leq \lambda\phi(a)+(1-\lambda)\phi(b)+(1-\lambda)\phi(a)+\lambda\phi(b) \\
&=\phi(a)+\phi(b).
\end{aligned}
$$ 
We would like to stress that this property could be also proved via Theorem \ref{main}, writing $\phi(c)+\phi(a+b-c) \leq \phi(b)+\phi(c+a+b-c-b)=\phi(b)+\phi(a)$. Nevertheless, by no means Theorem \ref{main} can be proved with a trivial argument using the convexity of the function.

\item If we wonder when we may guarantee strict inequality, $\phi(a)+\phi(b)<\phi(c)+\phi(a+b-c)$ when the three points are not equal to each other, we may take a look to Proposition \ref{eq_conv} and the inequalities (*) and (**) in the proof of Theorem \ref{main}. Therefore, if $\phi$ does not have any interval where its graph is linear and $a,b,c$ are three points not equal to each other (with $c \geq a,b$ or $c \leq a,b$), we would have that $\phi(a)+\phi(b) \lneq \phi(c)+\phi(a+b-c)$.

\item The condition of $\phi$ being bounded on a single interval can not be removed, since the counterexample presented in Proposition \ref{counterexample} serves also as a function that fulfills the inequality presented in Theorem \ref{main}, despite not being convex (not even continuous).

\end{enumerate}
\end{remark}

\begin{corollary} \label{cor1}
Let $I\subseteq \R$ be an interval and $\psi:I \rightarrow \mathbb{R}$ be a concave function. \\
Then, for every $c,\,a,\,b \in I$ so that $c \leq a,b$ (resp. $c \geq a,b$) and $a+b \in I$, we have that $\psi(a)+\psi(b)\geq \psi(c)+\phi(a+b-c)$. \\
If $\psi:\R \rightarrow \R$ is a locally bounded function that satisfies $\psi(a)+\psi(b)\geq \psi(c)+\phi(a+b-c)$ whenever $c \leq a,b$ or $c \geq a,b$, then $\psi$ is concave. 
\end{corollary}

\begin{proof}
Apply Theorem \ref{main} to the convex function $-\psi$.
\end{proof}

\section{Applications} \label{ineq}

A natural question that appears when regarding the proofs in Section \ref{main} is whether midconvexity is equivalent to the property proved in Theorem \ref{main}. Before giving an answer, we will need a preliminary inequality:

\begin{corollary} \label{cor2}
Let $a,b,c$ be numbers so that $c \leq a,b$ (or $c \geq a,b$). Then, $(a+b)c \leq c^2+ab$. If $a \leq c \leq b$, then $(a+b)c \geq c^2+ab$. If those three points are not equal to each other, the inequalities are strict.
\end{corollary}

\begin{proof}
Consider the function $\phi(x)=x^2$ in Theorem \ref{main} and Remark \ref{remark}. 
\end{proof}

\begin{corollary}
There exist functions which are midconvex but do not satisfy the inequality in Theorem \ref{main}.
\end{corollary}

\begin{proof}
Following the ideas reflected in the example given in Theorem \ref{counterexample}, consider a Hamel basis of $\R$ over $\Q$ $\{r_i\}_{i \in I}$ and fix $i_1, \, i_2 \in I$. Without loss of generality, we may assume that $r_{i_1}<r_{i_2}$ 
Define $f:\{r_i\}_{i \in I} \rightarrow \R$ as follows:
$$
f(r_i)=\begin{cases}
r_i & \text{ if }i \neq i_1, i_2, \\
-r_i+r_{i_1}+r_{i_2} & \text{ if $i=i_1$ or $i=i_2$,}
\end{cases}
$$
and extend $f$ to the whole $\R$ via linearity. We will show that then $g(x)=f(x)^2$ is a midconvex function which does not satisfy the inequality in Theorem \ref{main}. Indeed, $f(x)$ is an additive function, so the midconvexity of $g$ is trivial. Pick up now $i \in I$ so that $r_i>r_{i_2}$. Applying Corollary \ref{cor2}, we have that
$$
\begin{aligned}
\left[f(r_{i_2})+f(r_i)\right]f(r_{i_1})&=(r_{i_1}+r_i)r_{i_2} > r_{i_2}^2+r_{i_1}r_i=f^2(r_{i_1})+f(r_{i_2})f(r_i), \quad \text{from which} \\
0 &> 2f^2(r_{i_1})+2f(r_{i_2})f(r_i)-2f(r_{i_2})f(r_{i_1})-2f(r_i)f(r_{i_1}), \quad \text{ and then} \\
g(r_{i_2})+g(r_{i})&=f^2(r_{i_2})+f^2(r_{i}) \\
& > f^2(r_{i_1})+2f(r_{i_2})f(r_i)-2f(r_{i_2})f(r_{i_1})-2f(r_i)f(r_{i_1})+f^2(r_{i_1})+f^2(r_{i_2})+f^2(r_i)  \\
&=f^2(r_{i_1})+(f(r_{i})+f(r_{i_2})-f(r_{i_1}))^2=f^2(r_{i_1})+f^2(r_{i}+r_{i_2}-r_{i_1}) \\
&=g(r_{i_1})+g(r_{i}+r_{i_2}-r_{i_1}), 
\end{aligned}
$$
with $r_{i_1}<r_{i_2},r_i$.
\end{proof}

On the other hand, in regard of Theorem \ref{counterexample} and the subsequent comments, we remark that there can not be a nontrivial vector space all of whose nonzero elements are midconvex functions that do not satisfy the inequality in Theorem \ref{main} (we remind the reader that was inequality \eqref{ineq_main}). Indeed, first of all notice that we can not have a vector space of strictly midconvex functions (plus the zero function), since the negative of such an element would fulfill the opposite inequality. Therefore, a nonzero element $f$ of that nontrivial vector space would have to satisfy
\begin{equation} \label{mid_eq}
f\left({x_1+x_2 \over 2} \right)={1 \over 2}[f(x_1)+f(x_2)]
\end{equation}
for every $x_1,x_2 \in \R$. Also, if $f$ is a midconvex function not satisfying inequality \eqref{ineq_main}, then $f+k$ is also a midconvex function not satisfying inequality \eqref{ineq_main}, for every constant number $k \in \R$:
$$
\begin{aligned}
f\left({x_1+x_2 \over 2}\right)+k&={1\over 2}[f(x_1)+f(x_2)]+k={1\over 2}[(f(x_1)+k)+(f(x_2)+k)] \quad \text{and} \\
(f(a)+k)+(f(b)+k)&>f(c)+f(a+b-c)+2k=(f(c)+k)+(f(a+b-c)+k),
\end{aligned}
$$
for every $x_1,x_2,a,b,c \in \R$ so that $c \leq a,b$. Therefore, if 
$$
\left<f\right> \subseteq \Big\{\phi:\R \rightarrow \R \, : \, \phi \text{ is a midconvex function not satisfying the inequality in \eqref{ineq_main}}\Big\}\cup \{0\},
$$
we would have that 
$$
\left<f-f(0) \right> \subseteq \Big\{\phi:\R \rightarrow \R \, : \, \phi \text{ is a midconvex function not satisfying the inequality in \eqref{ineq_main}}\Big\}\cup \{0\}.
$$
 
\begin{lemma}
If $f$ is a function satisfying equation \eqref{mid_eq} for every $a,b \in \R$ and $f(0)=0$, then $f$ is additive.
\end{lemma}

\begin{proof} 
First of all, we can write $f(a)=f\left({2a+0 \over 2}\right)={1 \over 2}[f(2a)+f(0)]$, so that $f(2a)=2f(a)$. Therefore,
$$
f(a+b)=f\left({2a+2b \over 2} \right)={1 \over 2}(f(2a)+f(2b))=f(a)+f(b).
$$
\end{proof}

Hence, we would have that $f$ is rationally linear and in particular it would satisfy the inequality \eqref{ineq_main}. \\

We can also use inequality \eqref{ineq_main} for proving some other nontrivial inequalities for convex and concave functions.

\begin{theorem} \label{theorem2}
Let $\phi:\R \rightarrow \R$ be a convex increasing function  and $a,b \in [c,d]$ be such that $a+b<d+c$. \\
Then $\phi(a)+\phi(b) \leq \phi(c)+\phi(d)$.
\end{theorem}

\begin{proof}
Let $s \in \R$ be such that
$$
{b-a \over d-b}+1<s<{b-a \over a-c}+1
$$
(whose existence can be guaranteed from the condition $a-c<d-b$). Define the function
$$
\psi(t)=\begin{cases}
c & \text{if }t \leq 0, \\
(a-c)t+c & \text{if }0<t \leq 1, \\
{b-a \over s-1}(t-1)+a & \text{if }1<t \leq s, \\
(d-b)t+b-(d-b)s & \text{if }s <t \leq s+1, \\
2(d-b)(t-s-1)+d & \text{if }t >s+1.
\end{cases}
$$
Notice that $\psi$ is a polygonal function, with increasing slopes at each step, so that $\psi$ is a convex function that satisfies $\psi(0)=c, \, \psi(1)=a, \, \psi(s)=b$ and $\psi(s+1)=d$. Applying Lemma \ref{lem1} we would obtain that $\phi \circ \psi$ is convex and, via Theorem \ref{main} we would have that 
$$
\phi(a)+\phi(b)=\phi \circ \psi(1)+\phi \circ \psi(s) \leq \phi \circ \psi(0)+\phi \circ \psi(1+s)=\phi(c)+\phi(d),
$$
as desired.
\end{proof}

\begin{theorem} \label{theorem3}
Let $\phi:\mathbb{R} \rightarrow \mathbb{R}$ be a decreasing convex function and $a,b \in [c,d]$ be such that $a+b>c+d$. \\
Then $\phi(a)+\phi(b) \leq \phi(c)+\phi(d)$.
\end{theorem}

\begin{proof}
Similarly as in Theorem \ref{theorem2}, we may choose a number $r$ so that
$$
{b-a \over a-c}+1<r<{b-a \over d-b}+1
$$
and define the function
$$
\psi(t)=\begin{cases}
2(a-c)t+c & \text{if }t \leq 0, \\
(a-c)t+c & \text{if }0<t \leq 1, \\
{b-a \over r-1}(t-1)+a & \text{if }1<t \leq r, \\
(d-b)(t-r)+b & \text{if }r<t \leq r+1, \\
d & \text{if }r+1<t.
\end{cases}
$$
Then, $\psi$ is a polygonal function, with decreasing slopes at each step, so that $\psi$ is a concave function that satisfies $\psi(0)=c, \, \psi(1)=a, \, \psi(r)=b$ and $\psi(r+1)=d$. Applying again Lemma \ref{lem1}, $(-\phi \circ \psi)$ is a concave function so, via Corollary \ref{cor1},
$$
-(\phi(a)+\phi(b))=(-\phi \circ \psi)(1)+(-\phi \circ \psi)(r)  \geq (-\phi \circ \psi)(0)+(-\phi \circ \psi)(r+1)=-(\phi(c)+\phi(d)),
$$
as desired.
\end{proof}

\begin{corollary} \label{cor3}
If $0 < c \leq a,b$ (or $0<a,b \leq c \leq a+b$), then $\log(a+b-c) \leq \log(a)+\log(b)-\log(c)$.
\end{corollary}

\begin{proof}
Apply Corollary \ref{cor1} to the function $\phi(x)=\log(x)$. Notice that this corollary can also be deduced from Corollary \ref{cor2}, applying the increase of the logarithm to the inequality $a+b \leq c+{ab \over c}$, for $0<c \leq a,b$.
\end{proof}

\begin{corollary} \label{cor4}
Let $a,b$ be such that ${(k-1)\pi \over 2} \leq a,b,a+b \leq {k\pi \over 2}$ for some $k \in \Z$. \\
Then, $|\cos(a)|+|\cos(b)| \geq 1+ |\cos(a+b)|$.
\end{corollary}

\begin{corollary} \label{cor5}
Let $0 \leq \theta,\zeta \leq 1$ be so that $\theta+\zeta \geq 1$. Then, 
$$
\begin{aligned}
\theta\sqrt{1-\zeta^2}+\zeta\sqrt{1-\theta^2}&\leq \sqrt{(\theta+\zeta)(2-\theta-\zeta)}, \\
\sqrt{1-\theta^2}\sqrt{1-\zeta^2}+\theta+\zeta &\geq 1+\theta\zeta.
\end{aligned}
$$
\end{corollary}

\begin{proof}
Apply the Theorem \ref{main} to the function $\phi(x)=\arcsin(x)$ and choosing $a=\zeta, \, b=\theta$ and $c=1$ to obtain 
\begin{equation} \label{eq1}
 \arcsin(\theta)+\arcsin(\zeta) \leq {\pi \over 2}+\arcsin(\theta+\zeta-1).
\end{equation}
Taking sine in both sides, and taking into account that the function $\sin(x)$ is increasing over $[0,{\pi \over 2}]$, we obtain
$$
\begin{aligned}
\theta\sqrt{1-\zeta^2}+\zeta\sqrt{1-\theta^2}&=\sin(\arcsin(\theta)+\arcsin(\zeta)) \\
& \leq \sin\left({\pi \over 2}+\arcsin(\theta+\zeta-1)\right)=\sqrt{1-(\theta+\zeta-1)^2} \\
&=\sqrt{(\theta+\zeta)(2-\theta-\zeta)},
\end{aligned}
$$
obtaining the first inequality. \\
If we now take cosine in both sides of the equation \eqref{eq1}, now taking into account that $\cos(x)$ is decreasing over $[0,{\pi \over 2}]$, we obtain
$$
\begin{aligned}
\sqrt{1-\theta^2}\sqrt{1-\zeta^2}-\theta\zeta &=\cos(\arcsin(\theta)+\arcsin(\zeta))\\
&\geq \cos\left({\pi \over 2}+\arcsin(\theta+\zeta-1)\right)=1-(\theta+\zeta),
\end{aligned}
$$
obtaining the second inequality.
\end{proof}

\begin{corollary} \label{cor6}
Let $x,y$ be such that $x+y \geq 1$. Then,
$$
(x+y)^2 \geq {(1+x^2)(1+y^2) \over 1+(x+y)^2}.
$$
\end{corollary}

\begin{proof}
Apply Corollary \ref{cor1} to $\phi(t)=\arctan(t), \, a=x, \, b=y$ and $c=0$ to obtain $\arctan(x)+\arctan(y) \geq \arctan(x+y)$. Notice that, since $\phi$ is an increasing function, we would also have that $\arctan(x)+\arctan(y) \leq \arctan(x+y)+{\pi \over 2}$. Using that $\sin(\theta) \geq \cos(\zeta)$ whenever ${\pi \over 4} \leq \zeta \leq \theta \leq \zeta+{\pi \over 2}$ we can then write
$$
{x+y \over \sqrt{1+x^2}\sqrt{1+y^2}}=\sin(\arctan(x)+\arctan(y)) \geq \cos(\arctan(x+y))={1 \over \sqrt{1+(x+y)^2}}.
$$
\end{proof}

\begin{corollary} \label{cor7}
Let $0<c<a<b<d$. Then, $a+b>c+d$ if and only if $ab > cd$.
\end{corollary}

\begin{proof}
Apply Theorem \ref{theorem3} and Remark \ref{remark} to $\phi(x)=-\ln(x)$ and use the increasing property of the logarithmic function to show that if $a+b>c+d$, then $ab > cd$. \\
Reciprocally, observe that if $ab>cd$, then $\ln(a)+\ln(b)>\ln(c)+\ln(d)$. Apply Theorem \ref{theorem2} to the function $\phi(x)=e^x$ to obtain $a+b>c+d$.
\end{proof}

\section{The multivariable case} \label{multivariable}
We study here how we can use the inequality presented in Section \ref{main_sec} to obtain results for several variables:

\begin{theorem} \label{multi}
Let $V$ be a topological vector space, $U \subseteq V$ be a convex set, $\phi:U \rightarrow \R$ be a convex function and $a,b,c \in U$ be three aligned points that satisfy that $a+b-c$ is an element of $U$ as well and $c \notin [a,b]$. Then,
$$
\phi(a)+\phi(b) \leq \phi(c)+\phi(a+b-c).
$$
\end{theorem}

\begin{proof}
First of all, we remark that, since $a+b-c \in U$ and $c \notin [a,b]$, we obtain that $a,b \in [c,a+b-c]$, since we can write $c=(1-\lambda_0)a+\lambda_0 b$ (the points $a, b$ and $c$ are aligned) with $\lambda_0 \in \R \setminus [0,1]$ and the system
$$
\left(\begin{array}{c}c \\ a+b-c\end{array}\right)=\left[\begin{array}{cc}1-\lambda_0 & \lambda_0 \\ \lambda_0 & 1-\lambda_0\end{array}\right]\left(\begin{array}{c} a \\ b\end{array}\right)
$$
is equivalent to
$$
\left(\begin{array}{c} a \\ b\end{array}\right)=\left[\begin{array}{cc}{1-\lambda_0 \over 1-2\lambda_0} & 1-{1-\lambda_0 \over 1-2\lambda_0} \\ 1-{1-\lambda_0 \over 1-2\lambda_0} & {1-\lambda_0 \over 1-2\lambda_0}\end{array}\right]\left(\begin{array}{c}c \\ a+b-c\end{array}\right),
$$
where ${1-\lambda_0 \over 1-2\lambda_0} \in (0,1)$ if $\lambda_0 \in \R \setminus [0,1]$. Notice also that if $a=\lambda_1 c+(1-\lambda_1)(a+b-c)$ (for some $\lambda_1 \in (0,1)$), we obtain that $b=(1-\lambda_1)a+\lambda_1b$. Define 
$$
\begin{array}{rccl}
g:&[0,1]& \longrightarrow & U \\
& t & \mapsto & tc+(1-t)(a+b-c).
\end{array}
$$
Then, $\phi \circ g:[0,1] \rightarrow \R$ is a convex function. Applying Theorem \ref{main},
$$
\phi(a)+\phi(b)=(\phi \circ g)(\lambda_1)+(\phi \circ g)(1-\lambda_1) \leq (\phi \circ g)(0)+(\phi \circ g)(1)=\phi(c)+\phi(a+b-c).
$$
\end{proof}
If $(X, \|\cdot\|_X)$ is a normed space and $a,b$ and $c$ are elements of $X$, we can apply the inverse triangular inequality to write $\|a+b-c\|_X \geq \|a+b\|_X-\|c\|_X$. If those elements are aligned we can say something more:
\begin{corollary}
Let $(X,\| \cdot \|_X)$ be a normed space and $a,b, \, c$ be three aligned points, with $c \notin [a,b]$. Then
$$
\|a+b-c\|_X \geq \Big|\|a\|_X+\|b\|_X-\|c\|_X\Big|.
$$
In particular, if $\|a\|_X=\|b\|_X$ and $\lambda \in [0,1]$ we have that
$$
\|a+\lambda(a-b)\|_X \geq \|b+\lambda(a-b)\|_X.
$$
\end{corollary}

\begin{proof}
Since the norm in a normed space is a convex function, we can apply Theorem \ref{multi} to obtain $\|a\|_X+\|b\|_X \leq \|c\|_X+\|a+b-c\|_X$. \\
Reciprocally, the result is trivial for $a+b={c \over 2}$. Assume then $a+b \neq {c \over 2}$ and apply Theorem \ref{multi} with $\tilde{a}=\tilde{b}={c \over 2}$ and $\tilde{c}=a+b$ (since $a+b \notin \left[{c \over 2},{c \over 2}\right]$) to obtain
$$
\|c\|_X=\|{c \over 2}\|_X+\|{c \over 2}\|_X \leq \|a+b\|_X+\|{c \over 2}+{c \over 2}-(a+b)\|_X.
$$
If now $\|a\|_X=\|b\|_X$, notice that $a, b$ and $\lambda a+(1-\lambda)b$ are three aligned points and that $b \notin [a,\lambda a +(1-\lambda)b]$. Then,
$$
\|a+(\lambda a+(1-\lambda)b)-b\|_X+\|b\|_X \geq \|a\|_X+\|\lambda a+(1-\lambda)b\|_X, 
$$
from which the second part of the corollary follows.
\end{proof}


\end{document}